%% file: Minimal_Attached_Primes_of_Local_Cohomology_Modules_of_Block_Graphs.tex
\documentclass{article}
\input{preamble}

\title{Minimal Attached Primes of Local Cohomology Modules of Binomial Edge Ideals of Block Graphs}
\author{David Williams\footnote{The author was supported by EPSRC grant EP/T517835/1}}
\date{\vspace{-0.3cm}}

\makeatletter
\renewcommand*{\@fnsymbol}[1]{\@arabic{#1}}
\makeatother

\begin{document}

\maketitle

\begin{abstract}\label{Abstract}
	\noindent We calculate the minimal attached primes of the local cohomology modules of the binomial edge ideals of block graphs. In particular, we obtain a combinatorial characterisation of which of these modules are non-vanishing.

 	\vspace{0.35\baselineskip}
	\noindent We also show that the main result of this paper follows from a recent result of Lax, Rinaldo, and Romeo (\cite[Theorem 3.2]{laxSequentiallyCohenMacaulayBinomial2024}), which was published independently during the writing of this paper. This provides a short alternative proof of our result.
\end{abstract}

\begin{tcolorbox}
	Unless specified otherwise, let
	\begin{equation*}
		R=k[x_1,\ldots,x_n,y_1,\ldots,y_n]
	\end{equation*}
	for some field $k$ (of arbitrary characteristic) and $n\geq 2$. We denote by $\mathfrak{m}$ the irrelevant ideal of $R$. Furthermore, set $\delta_{i,j}\defeq x_iy_j-x_jy_i$ for $1\leq i<j\leq n$.

	Throughout this paper, all graphs will be finite, simple and undirected. For a graph $G$, we denote by $V(G)$ and $E(G)$ the sets of vertices and edges $\{i,j\}$ of $G$ respectively. Unless specified otherwise, all graphs will have at most $n$ vertices, with vertex set $\{1,\ldots,\smallAbs{V(G)}\}$.
\end{tcolorbox}

\section{Preliminaries}

\subsection{Binomial Edge Ideals}

Our objects of study are the binomial edge ideals of graphs. These  were introduced in \cite{herzogBinomialEdgeIdeals2010}, and independently in \cite{ohtaniGraphsIdealsGenerated2011}.

We begin with their definition (most of this subsection appeared previously in \cite[Section 1]{williamsLFCoversBinomialEdge2023}):

\begin{definition}\label{BEIDef}
	Let $G$ be a graph. We define the \emph{binomial edge ideal} of $G$ as
	\begin{equation*}
		\mathcal{J}(G)\defeq(\delta_{i,j}:\{i,j\}\in E(G))
	\end{equation*}
	We may also denote this ideal by $\mathcal{J}_G$.
\end{definition}

There is a rich interplay between the algebraic properties of these ideals and the combinatorial properties of the corresponding graph. For an overview of some of these interactions, see \cite[Chapter 7]{herzogBinomialIdeals2018}.

A key fact about binomial edge ideals is the following:

\begin{theorem}\thmCite[Corollary 2.2]{herzogBinomialEdgeIdeals2010}
	Binomial edge ideals are radical.
\end{theorem}

In fact, an explicit description of the primary decomposition can be given by purely combinatorial means, but we must first introduce some notation:

\begin{notation}
	Let $G$ be a graph, and $S\subseteq V(G)$. Denote by $c_G(S)$ the number of connected components of $G\setminus S$ (we will just write $c(S)$ when there is no confusion). Then we set
	\begin{equation*}
		\mathcal{C}(G)\defeq\{S\subseteq V(G):\text{\normalfont$S=\varnothing$ or $c(S\setminus\{v\})<c(S)$ for all $v\in S$}\}
	\end{equation*}
	That is, when we ``add back'' any vertex in $S\in\mathcal{C}(G)$, it must reconnect some components of $G\setminus S$.
\end{notation}

\begin{notation}
	Let $G$ be a graph, and $S\subseteq V(G)$. Denote by $G_1,\ldots,G_{c(S)}$ the connected components of $G\setminus S$, and let $\tilde{G}_i$ be the complete graph with vertex set $V(G_i)$. Then we set
	\begin{equation*}
		P_S(G)\defeq(x_i,y_i:i\in S)+\mathcal{J}(\tilde{G}_1)+\cdots+\mathcal{J}(\tilde{G}_{c(S)})
	\end{equation*}
\end{notation}

\begin{proposition}
	$\mathcal{J}(K_S)$ is prime for any $S\subseteq\{1,\ldots,n\}$, and therefore $P_T(G)$ is prime for any graph $G$ and $T\subseteq V(G)$ also.
\end{proposition}

\begin{proof}
	This follows from \cite[Theorem 2.10]{brunsDeterminantalRings1988}.
\end{proof}

\begin{theorem}\thmCite[Corollary 3.9]{herzogBinomialEdgeIdeals2010}\label{BEIPD}
	Let $G$ be a graph. Then
	\begin{equation*}
		\mathcal{J}(G)=\myCap_{S\in\mathcal{C}(G)}P_S(G)
	\end{equation*}
	is the primary decomposition of $\mathcal{J}(G)$.
\end{theorem}

We will also make use of the following:

\begin{lemma}\thmCite[Lemma 3.1]{herzogBinomialEdgeIdeals2010}\label{HeightLemma}
	Let $G$ be a graph, and $S\in\mathcal{C}(G)$. Then
	\begin{equation*}
		\hgt_R(P_S(G))=\smallAbs{S}+n-c(S)
	\end{equation*}
\end{lemma}

\subsection{Secondary Representations \& Attached Primes}

Secondary representations were introduced independently by \cite{kirbyCoprimaryDecompositionArtinian1973} (there called ``coprimary decompositions''), Macdonald in \cite{macdonaldSecondaryRepresentationModules1973}, and Moore in \cite{moorePrimaryCoprimaryDecompositions1973} (again as ``coprimary decompositions''), with Kirby's being the first paper received.

We begin with their definitions:

\begin{definition}
	Let $R$ be a ring. We say that a non-zero $R$-module $S$ is \emph{secondary} if, for each $r\in R$, either $rS=S$ or $r^mS=0$ for some $m\geq1$.

	In this case, $\mathfrak{p}=\sqrt{\Ann_R(S)}$ is prime, and we say that $S$ is \emph{$\mathfrak{p}$-secondary}.
\end{definition}

\begin{definition}
	Let $R$ be a ring and $M$ an $R$-module. If there exist $\mathfrak{p}_i$-secondary $R$-submodules $S_i$ of $M$ such that
	\begin{equation*}
		M=S_1+\cdots+S_t
	\end{equation*}
	for some $t\geq1$, then we call this a \emph{secondary representation} of $M$, and $M$ is said to be \emph{representable}.

	Such a representation is said to be \emph{minimal} when
	\begin{enumerate}[label=\roman*)]
		\item The $\mathfrak{p}_i$ are all distinct.
		\item For every $1\leq i\leq t$, we have
		\begin{equation*}
			S_i\nsubseteq\sum_{\mathclap{\substack{j=1\\j\neq i}}}^tS_j
		\end{equation*}
	\end{enumerate}
\end{definition}

\begin{note}
	It is easily shown that the sum of any two $\mathfrak{p}$-secondary modules is $\mathfrak{p}$-secondary, and so any secondary representation can be refined to be minimal.
\end{note}

There are several uniqueness theorems concerning secondary representations, we will need only one for our purposes:

\begin{theorem}\thmCite[2.2]{macdonaldSecondaryRepresentationModules1973}
	Let $R$ be a ring, $M$ a representable $R$-module, and $S_i$ the $\mathfrak{p}_i$-secondary summands in a minimal secondary representation of $M$ for some $t\geq1$. Then both $t$ and the set $\{\mathfrak{p}_1,\ldots,\mathfrak{p}_t\}$ are independent of the choice of minimal secondary representation.
\end{theorem}

\begin{definition}
	Let $R$ be a ring, $M$ a representable $R$-module, and $S_i$ the $\mathfrak{p}_i$-secondary summands in a minimal secondary representation of $M$ for some $t\geq1$. Then we set
	\begin{equation*}
		\Att_R(M)=\{\mathfrak{p}_1,\ldots,\mathfrak{p}_t\}
	\end{equation*}
	and these are said to be the \emph{attached primes} of $M$.

	If an attached prime is minimal (with respect to inclusion) in this set, we say that it is \emph{isolated}, otherwise we say that it is \emph{embedded}.
\end{definition}

The theory of secondary representations is in some sense dual to that of primary decompositions. For example, in the same way that every ideal of a Noetherian ring has a primary decomposition, we have the following:

\begin{theorem}\thmCite[(5.2)]{macdonaldSecondaryRepresentationModules1973}\label{ArtModRep}
	Let $R$ be a ring and $M$ an Artinian $R$-module. Then $M$ is representable.
\end{theorem}

Furthermore, in the same way that for an $R$-module $M$ we have that $\mathfrak{p}\in\Ass_R(M)$ if and only if $R/\mathfrak{p}\hookrightarrow M$, the following is true of attached primes (for Noetherian rings):

\begin{lemma}\thmCite[(2.5)]{macdonaldSecondaryRepresentationModules1973}\label{AttSurjLemma}
	Let $R$ be a ring and $M$ a representable $R$-module. Then $\mathfrak{p}\in\Att_R(M)$ if and only if there exists a $R$-module $N$ such that $\Ann_R(N)=\mathfrak{p}$ and an $R$-epimorphism $M\twoheadrightarrow N$.
\end{lemma}

\label{AssAttDual}This duality can be made explicit, at least for Noetherian local rings, via the Matlis Duality functor, which was introduced by Matlis in \cite{matlisInjectiveModulesNoetherian1958} (this is variously denoted $-^\vee$ or $D(-)$, we adopt $-^\vee$ here).

To be precise, if $R$ is a Noetherian complete local ring and $M$ a non-zero finitely generated $R$-module, then
\begin{equation*}\label{LocalAssAttDual}
	\Ass_R(M)=\Att_R(M^\vee)
\end{equation*}
See \cite[Corollary 10.2.20]{brodmannLocalCohomologyAlgebraic2013} for the details.

In our case, $R$ is not a local ring. However it is Noetherian, Gorenstein, and non-negatively $\mathbb{Z}$-graded, with a unique homogeneous maximal ideal. These properties allow the above result, as well as several others usually stated for complete local rings, to hold for graded $R$-modules (such as quotients by binomial edge ideals). We will provide in \hyperref[GradedAppendix]{the appendix} proofs (or references) for the results of this kind which we will use in this paper.

\section{A Short Exact Sequence for Block Graphs}\label{BGESSection}

We begin with some definitions:
\begin{definition}
	For a graph $G$ and vertex $v\in V(G)$, we denote by $N_G(v)$ the \emph{neighbourhood} of $v$ in $G$, that is, the vertices $w$ in $G$ for which $\{v,w\}\in E(G)$ (recall that our edges are undirected). Furthermore, we denote by $N_G[v]$ the \emph{closed neighbourhood} of $v$ in $G$, which is given by $N_G(v)\cup\{v\}$.
\end{definition}

\begin{definition}
	We say that a vertex $v$ of a graph $G$ is a \emph{cut vertex} of $G$ if the induced subgraph of $G$ obtained by removing $v$ has a greater number of connected components than $G$.
\end{definition}

\begin{definition}
	A graph is said to be \emph{biconnected} if it is connected, and remains connected if any one of its vertices is removed. Maximal biconnected subgraphs of a graph $G$ are called \emph{biconnected components} of $G$. We say that a graph $G$ is a \emph{block graph} if every biconnected component of $G$ is a clique.
\end{definition}

\begin{definition}\label{LeafCliqueDef}
	We say that a maximal clique of a block graph $G$ is a \emph{leaf clique} of $G$ if it contains at most one vertex which intersects with another maximal clique, or a \emph{branch clique} of $G$ otherwise.
\end{definition}

\begin{note}
	\cref{LeafCliqueDef} is not established terminology.
\end{note}

Next, we introduce some notation:
\begin{notation}
	For any cut vertex $v$ of $G$, we denote by $G_v$ the graph obtained by adding the edges of the complete graph with vertex set $N_G(v)$ to $G$ (that is, we complete the neighbourhood of $v$ in $G$).
\end{notation}

Suppose that $G$ has a cut vertex. By \cite[Lemma 4.8]{ohtaniGraphsIdealsGenerated2011}, we have
\begin{equation*}
	\mathcal{J}_G=\mathcal{J}_{G_v}\cap(\mathcal{J}_G+(x_v,y_v))
\end{equation*}
Note that
\begin{equation*}
	\mathcal{J}_G+(x_v,y_v)=\mathcal{J}_{G\setminus\{v\}}+(x_v,y_v)
\end{equation*}
and
\begin{equation*}
	\mathcal{J}_{G_v}+(\mathcal{J}_G+(x_v,y_v))=\mathcal{J}_{G_v\setminus\{v\}}+(x_v,y_v)
\end{equation*}
For brevity, we set
\begin{align*}
	G'&=G_v\\
	G''&=G\setminus\{v\}\\
	H&=G_{v}\setminus\{v\}
\end{align*}
and
\begin{align*}
	Q_1&=\mathcal{J}_{G'}\\
	Q_2&=\mathcal{J}_{G''}+(x_v,y_v)\\
	Q_3&=\mathcal{J}_H+(x_v,y_v)
\end{align*}
Then we obtain the short exact sequence
\begin{equation*}\label{BGSES}
	\begin{tikzcd}
		0 \arrow[r] & R/\mathcal{J}_G \arrow[r,"",hook] & R/Q_1\oplus R/Q_2 \arrow[r,"",two heads] & R/Q_3 \arrow[r] & 0
	\end{tikzcd}
\end{equation*}

\begin{example}
	If
	\begin{equation*}
		G=\quad\begin{tikzpicture}[x=0.7cm,y=0.7cm,every node/.style={circle,draw=black,fill=black,inner sep=0pt,minimum size=5pt},label distance=0.15cm,line width=0.25mm,baseline={([yshift=-0.5ex]current bounding box.center)}]
        	\node(1) at (0,0) {};
			\node(2) at (0,1) {};
			\node(3) at (1,1) {};
			\node(4) at (1,0) {};
			\node[label={[label distance=0.05cm]180:$a$}](5) at ([shift=(135:1)]2) {};
			\node(6) at ([shift=(75:1)]3) {};
			\node[label={[label distance=0.05cm]30:$v$}](7) at ([shift=(15:1)]3) {};
			\node(8) at ([shift=(-15:1)]7) {};
			\node[label={[label distance=0.02cm]0:$b$}](9) at ([shift=(15:1)]8) {};
			\node(10) at ([shift=(-45:1)]8) {};

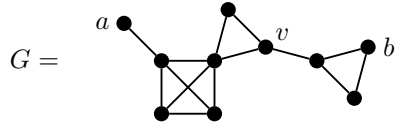
	    	\foreach \from/\to in {1/2,1/3,1/4,2/3,2/4,2/5,3/4,3/6,3/7,6/7,7/8,8/9,8/10,9/10}
        		\draw[-] (\from) -- (\to);
    	\end{tikzpicture}
	\end{equation*}
	then the cliques containing the vertices $a$ and $b$ are leaf cliques, will all other cliques being branch cliques, and
	\begin{equation*}
		G'=\quad
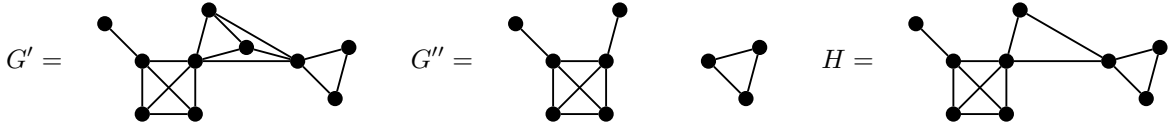
\begin{tikzpicture}[x=0.7cm,y=0.7cm,every node/.style={circle,draw=black,fill=black,inner sep=0pt,minimum size=5pt},label distance=0.15cm,line width=0.25mm,baseline={([yshift=-0.5ex]current bounding box.center)}]
        	\node(1) at (0,0) {};
			\node(2) at (0,1) {};
			\node(3) at (1,1) {};
			\node(4) at (1,0) {};
			\node(5) at ([shift=(135:1)]2) {};
			\node(6) at ([shift=(75:1)]3) {};
			\node(7) at ([shift=(15:1)]3) {};
			\node(8) at ([shift=(-15:1)]7) {};
			\node(9) at ([shift=(15:1)]8) {};
			\node(10) at ([shift=(-45:1)]8) {};

	    	\foreach \from/\to in {1/2,1/3,1/4,2/3,2/4,2/5,3/4,3/6,3/7,3/8,6/7,6/8,7/8,8/9,8/10,9/10}
        		\draw[-] (\from) -- (\to);
		\end{tikzpicture}\quad\quad G''=\quad\begin{tikzpicture}[x=0.7cm,y=0.7cm,every node/.style={circle,draw=black,fill=black,inner sep=0pt,minimum size=5pt},label distance=0.15cm,line width=0.25mm,baseline={([yshift=-0.5ex]current bounding box.center)}]
        	\node(1) at (0,0) {};
			\node(2) at (0,1) {};
			\node(3) at (1,1) {};
			\node(4) at (1,0) {};
			\node(5) at ([shift=(135:1)]2) {};
			\node(6) at ([shift=(75:1)]3) {};
			\node(8) at ([shift=(-15:1)]7) {};
			\node(9) at ([shift=(15:1)]8) {};
			\node(10) at ([shift=(-45:1)]8) {};

	    	\foreach \from/\to in {1/2,1/3,1/4,2/3,2/4,2/5,3/4,3/6,8/9,8/10,9/10}
        		\draw[-] (\from) -- (\to);
    	\end{tikzpicture}\quad\quad H=\quad\begin{tikzpicture}[x=0.7cm,y=0.7cm,every node/.style={circle,draw=black,fill=black,inner sep=0pt,minimum size=5pt},label distance=0.15cm,line width=0.25mm,baseline={([yshift=-0.5ex]current bounding box.center)}]
        	\node(1) at (0,0) {};
			\node(2) at (0,1) {};
			\node(3) at (1,1) {};
			\node(4) at (1,0) {};
			\node(5) at ([shift=(135:1)]2) {};
			\node(6) at ([shift=(75:1)]3) {};
			\node(8) at ([shift=(-15:1)]7) {};
			\node(9) at ([shift=(15:1)]8) {};
			\node(10) at ([shift=(-45:1)]8) {};

	    	\foreach \from/\to in {1/2,1/3,1/4,2/3,2/4,2/5,3/4,3/6,3/8,6/8,8/9,8/10,9/10}
        		\draw[-] (\from) -- (\to);
		\end{tikzpicture}
	\end{equation*}
\end{example}

\section{Some Properties of Cut Vertices of Block Graphs}\label{CVPropsSection}

\begin{tcolorbox}
	Throughout this \lcnamecref{CVPropsSection}, we suppose that $G$ is a block graph on $n$ vertices which is not a disjoint union of cliques, and so it has at least one cut vertex.

	$v$ will always denote a cut vertex of $G$, and we adopt the notation of \cref{BGESSection}.
\end{tcolorbox}

We will next compute $\mathcal{C}(G')$, as well as $\mathcal{C}(G'')$ and $\mathcal{C}(H)$ for \hyperref[GoodCVExistence]{certain $v$}, relative to $\mathcal{C}(G)$.

We begin with a preliminary \lcnamecref{PathEquiv}:
\begin{lemma}\label{PathEquiv}
	For any $S\subseteq V(G)$ with $v\notin S$, and any vertices $a$ and $b$ of $G\setminus S$ other than $v$, the following are equivalent:
	\begin{enumerate}
		\item\label{PathEquiv1} $a$ and $b$ are connected by a path in $G\setminus S$.
		\item\label{PathEquiv2} $a$ and $b$ are connected by a path in $G'\setminus S$.
		\item\label{PathEquiv3} $a$ and $b$ are connected by a path in $H\setminus S$.
	\end{enumerate}
\end{lemma}

\begin{proof}\mbox{}
	\newlength{\mylength}\maxlength{\ref{PathEquivProof1},\ref{PathEquivProof2},\ref{PathEquivProof3}}{\mylength}
	\begin{enumerate}[labelindent=0pt,labelwidth=\mylength,leftmargin=!]
		\myitem{{\hyperref[PathEquiv1]{(1)} $\Rightarrow$ \hyperref[PathEquiv2]{(2)}:}}\label{PathEquivProof1} This follows immediately from the fact that $E(G')\supseteq E(G)$.
		\myitem{{\hyperref[PathEquiv2]{(2)} $\Rightarrow$ \hyperref[PathEquiv3]{(3)}:}}\label{PathEquivProof2} Let $P$ be a path connecting $a$ and $b$ in $G'\setminus S$. Since $H=G'\setminus\{v\}$, the only issue that may arise is if $P$ passes through $v$. Then suppose that this is the case, so $P$ contains edges $\{w_1,v\}$ and $\{v,w_2\}$ for some $w_1,w_2\in N_G(v)\setminus S$. Since $N_G(v)$ has been completed in $H$, we may replace $\{w_1,v\}$ and $\{v,w_2\}$ in $P$ with $\{w_1,w_2\}$. This new path then connects $a$ and $b$ in $H\setminus S$.
	\myitem{{\hyperref[PathEquiv3]{(3)} $\Rightarrow$ \hyperref[PathEquiv1]{(1)}:}}\label{PathEquivProof3} Now let $P$ be a path connecting $a$ and $b$ in $H\setminus S$. Since $H$ is obtained by completing $N_G(v)$ in $G$ and then removing $v$, the only issue that may arise here is if $P$ includes any edges of the form $\{w_1,w_2\}$ for some $w_1,w_2\in N_G(v)\setminus S$ with $\{w_1,w_2\}\notin E(G)$, so suppose that this is the case. Since $v\notin S$, we know that the edges $\{w_1,v\}$ and $\{v,w_2\}$ belong to $G\setminus S$, and so we may use these edges to replace $\{w_1,w_2\}$ in $P$. If $P$ contains several such ``bad'' edges, the graph obtained by making these replacements in $P$, say $Q$, will no longer be a path. However, it will still be connected, and so we can find a subgraph of $Q$ which is a path connecting $a$ and $b$ in $G\setminus S$.\qedhere
	\end{enumerate}
\end{proof}

We can now compute $\mathcal{C}(G')$:

\begin{proposition}\label{GPrimeSepSets}
	We have
	\begin{equation*}
		\mathcal{C}(G')=\{S\in\mathcal{C}(G):v\notin S\}
	\end{equation*}
\end{proposition}

\begin{proof}
	We will first show that
	\begin{equation*}
		\mathcal{C}(G')\supseteq\{S\in\mathcal{C}(G):v\notin S\}
	\end{equation*}
	Take any $S\in\mathcal{C}(G)$ such that $v\notin S$, and any $w\in S$. To show that $S\in\mathcal{C}(G')$, we must show that adding $w$ back to $G'\setminus S$ reconnects (at least) two vertices lying in separate connected components of $G'\setminus S$. We know that adding $w$ back to $G\setminus S$ reconnects at least two vertices $a,b\in N_G(w)\setminus S$ lying in separate connected components of $G\setminus S$ since $S\in\mathcal{C}(G)$. By \cref{PathEquiv}, $a$ and $b$ will lie in separate connected components of $G'\setminus S$ also, and will be reconnected in $G'\setminus S$ by adding $w$ back to $G'\setminus S$. Then $S\in\mathcal{C}(G')$, and the first inclusion follows.

	We will now show that
	\begin{equation*}
		\mathcal{C}(G')\subseteq\{S\in\mathcal{C}(G):v\notin S\}
	\end{equation*}
	Take any $S\in\mathcal{C}(G')$, and any $w\in S$. Note that any vertices in $N_G(v)\setminus S$ will be connected in $G'\setminus S$ (since we obtained $G'$ from $G$ by completing $N_G(v)$), and so we cannot have $v\in S$, since adding it back to $G'\setminus S$ cannot then reconnect any separate connected components of $G'\setminus S$, which would contradict that $S\in\mathcal{C}(G')$. 

	To show that $S\in\mathcal{C}(G)$, we must show that adding $w$ back to $G\setminus S$ reconnects (at least) two vertices lying in separate connected components of $G\setminus S$. We know that adding $w$ back to $G'\setminus S$ reconnects at least two vertices $a,b\in N_{G'}(w)\setminus S$ lying in separate connected components of $G'\setminus S$ since $S\in\mathcal{C}(G')$. By \cref{PathEquiv}, $a$ and $b$ will lie in separate connected components of $G\setminus S$ also, and will be reconnected in $G\setminus S$ by adding $w$ back to $G\setminus S$. Then $S\in\mathcal{C}(G)$, and the result follows.
\end{proof}

We will make use of the following two lemmas:

\begin{lemma}\label{ExtremalLeafCliqueExist}
	Suppose that $G$ has at least two cut vertices. Then there exists a leaf clique of $G$ which intersects with exactly one branch clique of $G$.
\end{lemma}

\begin{proof}
	The proof will be by induction on the number of cut vertices of $G$. If $G$ has a exactly two cut vertices then the result is obvious, so suppose that $G$ has more than two cut vertices.

	Since $G$ has a cut vertex, we may choose a leaf clique $L$ of $G$ with a cut vertex $w\in V(L)$, and set $A=G\setminus(C\setminus\{w\})$. That is, we obtain $A$ by removing $L$ from $G$, but keeping $w$. $A$ has fewer cut vertices than $G$, and so we can find a leaf clique $C$ of $A$ which intersects with exactly one branch clique of $A$ by induction.

	If $C$ remains a leaf clique in $G$, then it must intersect with exactly one branch clique of $G$, since we only removed a leaf clique from $G$ to obtain $A$. The only way for $C$ to become a branch clique in $G$ is if $L$ is a leaf clique intersecting with it. In this case, since $C$ is a leaf clique in $A$, the only branch clique that $L$ can intersect with in $G$ is $C$. Then, in either case, we can find a leaf clique of $G$ which intersects with exactly one branch clique of $G$, so we are done.
\end{proof}

\begin{lemma}\label{SepSetContainment}
	Suppose that $S\in\mathcal{C}(H)$, and that $N_G(v)\subseteq S$. Then $S\in\mathcal{C}(G'')$.
\end{lemma}

\begin{proof}
	Since $N_G(v)\subseteq S$, we have $H\setminus S=G''\setminus S$. The only edges in $H$ that are not in $G''$ are between neighbours of $v$, but these are all removed in $H\setminus S$, so adding any single vertex back to either $H\setminus S$ or $G''\setminus S$ will have the same effect.
\end{proof}

For the purposes of \cref{MainBGThmProofSection}, we will need to choose $v$ with a particular property:

\begin{proposition}\label{GoodCVExistence}
	There exists a cut vertex $v$ of $G$ such that
	\begin{equation*}\label{ClaimedEquality}
		\mathcal{C}(G'')=\{S\setminus\{v\}:\text{\normalfont $S\in\mathcal{C}(G)$ with $v\in S$}\}\tag*{$(\dag)$}
	\end{equation*}
\end{proposition}

\begin{proof}
	The inclusion
	\begin{equation*}
		\mathcal{C}(G'')\supseteq\{S\setminus\{v\}:\text{$S\in\mathcal{C}(G)$ with $v\in S$}\}
	\end{equation*}
	clearly holds for any cut vertex $v$ of $G$, and so we will now find a cut vertex of $G$ satisfying the reverse inclusion.

	We proceed by induction on the number of cut vertices of $G$. If $G$ has a single cut vertex then the result is obvious, so suppose that $G$ has more than one cut vertex.

	By \cref{ExtremalLeafCliqueExist}, we can choose a leaf clique $L$ of $G$, with cut vertex $w\in V(L)$, which intersects with exactly one branch clique $B$ of $G$.

	If this $w$ belongs to at least two leaf cliques, then, for any $S\in\mathcal{C}(G'')$, we have $S\cup\{w\}\in\mathcal{C}(G)$, since adding $w$ back to $G\setminus S$ will reconnect these cliques, and so the desired inclusion would be satisfied by taking $v=w$.

	Otherwise, $w$ belongs to a single leaf clique, and so we are in the situation
	\begin{equation*}
		G=\quad
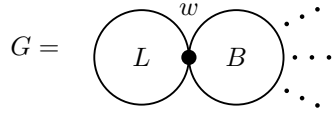
\begin{tikzpicture}[x=1.5cm,y=1.5cm,every node/.style={circle,draw=black,fill=black,inner sep=0pt,minimum size=5pt},label distance=0.15cm,line width=0.25mm,baseline={([yshift=-0.5ex]current bounding box.center)}]

			\node[style={circle,draw=black,fill=none,inner sep=0pt,minimum size=1.25cm}](L) at (0,0) {$L$};
            \node[label={[label distance=0.375cm]90:$w$}](w) at ([shift=(0:0.625cm)]L) {};
			\node[style={circle,draw=black,fill=none,inner sep=0pt,minimum size=1.25cm}](B) at ([shift=(0:1.25cm)]L) {$B$};
			\node[style={draw=none,fill=none},label={}](a) at ([shift=(0:0.625cm)]B) {};
			\node[style={draw=none,fill=none},label={}](b) at ([shift=(45:0.625cm)]B) {};
			\node[style={draw=none,fill=none},label={}](c) at ([shift=(-45:0.625cm)]B) {};

			\tikzset{little dot/.style={circle,draw=black,fill=black,inner sep=0pt,minimum size=1pt}}

			\node[little dot](a1) at ([shift=(0:0.1)]a) {};
			\node[little dot](a2) at ([shift=(0:0.15)]a1) {};
			\node[little dot](a3) at ([shift=(0:0.15)]a2) {};

			\node[little dot](b1) at ([shift=(0:0.15)]b) {};
			\node[little dot](b2) at ([shift=(25.5:0.15)]b1) {};
			\node[little dot](b3) at ([shift=(25.5:0.15)]b2) {};

			\node[little dot](c1) at ([shift=(0:0.15)]c) {};
			\node[little dot](c2) at ([shift=(-25.5:0.15)]c1) {};
			\node[little dot](c3) at ([shift=(-25.5:0.15)]c2) {};
    	\end{tikzpicture}
	\end{equation*}
	(where circles denote cliques).

	As in \cref{ExtremalLeafCliqueExist}, set $A=G\setminus(L\setminus\{w\})$. $A$ has fewer cut vertices than $G$, and so we can find some cut vertex $v$ of $A$ satisfying \ref{ClaimedEquality} for $A$ by induction. We claim that this $v$ also satisfies \ref{ClaimedEquality} for $G$ itself.
	
	Take any $S\in\mathcal{C}(G'')$. We aim to show that $S\cup\{v\}\in\mathcal{C}(G)$.

	Now, clearly $S\setminus\{w\}\in\mathcal{C}(A'')$ (note that we do not necessarily have $w\in S$), and so by the inductive hypothesis we have
	\begin{equation*}\label{SWVInCG}
		(S\setminus\{w\})\cup\{v\}\in\mathcal{C}(A)\subseteq\mathcal{C}(G)\tag*{$(\diamond)$}
	\end{equation*}
	Then if $w\notin S$ we are done, so suppose that $w\in S$.

	For any $u\in S$, adding $u$ back to $G''\setminus S$ reconnects at least two vertices $a_u,b_u\in N_{G''}(u)\setminus S$ lying in separate connected connected components of $G''\setminus S$ since $S\in\mathcal{C}(G'')$. Neither $a_u$ nor $b_u$ can be $v$ since $v\notin V(G'')$, and so
	\begin{equation*}
		a_u,b_u\in N_G(u)\setminus(S\cup\{v\})
	\end{equation*}
	Furthermore, we have
	\begin{equation*}
		G''\setminus S=G\setminus(S\cup\{v\})
	\end{equation*}
	so $a_u$ and $b_u$ will also lie in separate connected components of $G\setminus(S\cup\{v\})$, and will still be reconnected by adding $u$ back to $G\setminus(S\cup\{v\})$.

	Note in particular that, since $w\in S$, we have at least two vertices
	\begin{equation*}
		a_w,b_w\in N_G(w)\setminus(S\cup\{v\})
	\end{equation*}
	and since $w$ belongs to only two cliques, $L$ and $B$, in $G$, we may assume (without loss of generality) that $a_w\in V(L)$ and $b_w\in V(B)$, as $a_w$ and $b_w$ must lie in separate connected components of $G''\setminus S$.

	To conclude the proof, we wish to show that adding $v$ back to $G\setminus(S\cup\{v\})$ reconnects (at least) two vertices lying in separate connected components of $G\setminus(S\cup\{v\})$.

	For brevity, let $U=(S\setminus\{w\})\cup\{v\}$. We saw in \ref{SWVInCG} that $U\in\mathcal{C}(G)$, and so adding $v$ back to $G\setminus U$ reconnects at least two vertices $a_v,b_v\in N_G(v)\setminus U$ lying in separate connected connected components of $G\setminus U$. Note that neither $a_v$ nor $b_v$ can belong to $L$, since $v\neq w$ and $w$ is the only cut vertex belonging to $L$.

	Removing $w$ from $G\setminus U$ simply disconnects $B$ and $L$, so, if neither $a_v$ nor $b_v$ is $w$, $a_v$ and $b_v$ will trivially lie in separate connected components of $G\setminus(S\cup\{v\})$, and will still be reconnected by adding $v$ back to $G\setminus(S\cup\{v\})$.

	However, if we have say $b_v=w$, then $b_v\notin G\setminus(S\cup\{v\})$, and so the result does not immediately follow. In this case, we have $v\in V(B)$, since the only branch clique that $w$ belongs to is $B$, and so every cut vertex that $w$ is adjacent to must belong to $B$ also. Then we are in the following situation:
	\begin{equation*}
		G=\quad
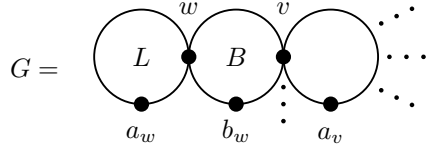
\begin{tikzpicture}[x=1.5cm,y=1.5cm,every node/.style={circle,draw=black,fill=black,inner sep=0pt,minimum size=5pt},label distance=0.15cm,line width=0.25mm,baseline={([yshift=-0.5ex]current bounding box.center)}]

			\node[style={circle,draw=black,fill=none,inner sep=0pt,minimum size=1.25cm}](L) at (0,0) {$L$};
            \node[label={[label distance=0.375cm]90:$w$}](w) at ([shift=(0:0.625cm)]L) {};
			\node[label={[label distance=0.06cm]-90:$a_w$}](aw) at ([shift=(-90:0.625cm)]L) {};
			\node[style={circle,draw=black,fill=none,inner sep=0pt,minimum size=1.25cm}](B) at ([shift=(0:1.25cm)]L) {$B$};
			\node[label={[label distance=0.41cm]90:$v$}](v) at ([shift=(0:0.625cm)]B) {};
			\node[label={[label distance=0cm]-90:$b_w$}](bw) at ([shift=(-90:0.625cm)]B) {};
			\node[style={circle,draw=black,fill=none,inner sep=0pt,minimum size=1.25cm}](C) at ([shift=(0:1.25cm)]B) {};
			\node[label={[label distance=0.0925cm]-90:$a_v$}](av) at ([shift=(-90:0.625cm)]C) {};

			\node[style={draw=none,fill=none},label={}](a) at ([shift=(0:0.625cm)]C) {};
			\node[style={draw=none,fill=none},label={}](b) at ([shift=(45:0.625cm)]C) {};
			\node[style={draw=none,fill=none},label={}](c) at ([shift=(-45:0.625cm)]C) {};

			\tikzset{little dot/.style={circle,draw=black,fill=black,inner sep=0pt,minimum size=1pt}}

			\node[little dot](a1) at ([shift=(0:0.1)]a) {};
			\node[little dot](a2) at ([shift=(0:0.15)]a1) {};
			\node[little dot](a3) at ([shift=(0:0.15)]a2) {};

			\node[little dot](b1) at ([shift=(0:0.15)]b) {};
			\node[little dot](b2) at ([shift=(25.5:0.15)]b1) {};
			\node[little dot](b3) at ([shift=(25.5:0.15)]b2) {};

			\node[little dot](c1) at ([shift=(0:0.15)]c) {};
			\node[little dot](c2) at ([shift=(-25.5:0.15)]c1) {};
			\node[little dot](c3) at ([shift=(-25.5:0.15)]c2) {};

			\node[little dot](v1) at ([shift=(-90:0.4cm)]v) {};
			\node[little dot](v2) at ([shift=(-90:0.15)]v1) {};
			\node[little dot](v3) at ([shift=(-90:0.15)]v2) {};

    	\end{tikzpicture}
	\end{equation*}
	We may then instead consider $a_v$ and $b_w\in V(B)\setminus(S\cup\{v\})$. We have $b_w\in N_G(v)\setminus(S\cup\{v\})$, $a_v$ and $b_w$ will clearly lie in separate connected components of $G\setminus(S\cup\{v\})$, and adding $v$ back to $G\setminus(S\cup\{v\})$ reconnects $a_v$ and $b_w$. Then $S\cup\{v\}\in\mathcal{C}(G)$, and we are done.
\end{proof}

We can compute $\mathcal{C}(H)$ for such a $v$:

\begin{proposition}\label{HSepSets}
	Let $v$ be as in \cref{GoodCVExistence}. Then we have
	\begin{equation*}
		\mathcal{C}(H)=\{S\in\mathcal{C}(G):\text{\normalfont $v\notin S$ and $N_G(v)\nsubseteq S$}\}
	\end{equation*}
\end{proposition}

\begin{proof}
	We will first show that
	\begin{equation*}
		\mathcal{C}(H)\supseteq\{S\in\mathcal{C}(G):\text{$v\notin S$ and $N_G(v)\nsubseteq S$}\}
	\end{equation*}
	Take any $S\in\mathcal{C}(G)$ such that $v\notin S$ and $N_G(v)\nsubseteq S$, and any $w\in S$. To show that $S\in\mathcal{C}(H)$, we must show that adding $w$ back to $H\setminus S$ reconnects (at least) two vertices lying in separate connected components of $H\setminus S$. We know that adding $w$ back to $G\setminus S$ reconnects at least two vertices $a,b\in N_G(w)\setminus S$ lying in separate connected components of $G\setminus S$ since $S\in\mathcal{C}(G)$. We may assume that neither $a$ nor $b$ is $v$, since $N_G(v)\nsubseteq S$ and so we may replace $v$ with a neighbour if necessary. By \cref{PathEquiv}, $a$ and $b$ will lie in separate connected components of $H\setminus S$ also, and will be reconnected in $H\setminus S$ by adding $w$ back to $H\setminus S$. Then $S\in\mathcal{C}(H)$, and the first inclusion follows.

	We will now show that
	\begin{equation*}
		\mathcal{C}(H)\subseteq\{S\in\mathcal{C}(G):\text{$v\notin S$ and $N_G(v)\nsubseteq S$}\}
	\end{equation*}
	Take any $S\in\mathcal{C}(H)$. Trivially, $v\notin S$. If $N_G(v)\subseteq S$, then, by \cref{SepSetContainment}, we have $S\in\mathcal{C}(G'')$. Since $v$ is as in \cref{GoodCVExistence}, we then have $S\cup\{v\}\in\mathcal{C}(G)$. But $N_G[v]\subseteq S\cup\{v\}$, so this clearly cannot belong to $\mathcal{C}(G)$, since adding $v$ back to $G\setminus(S\cup\{v\})$ would not reconnect any separate connected components of $G\setminus(S\cup\{v\})$. Then we must have $N_G(v)\nsubseteq S$.

	Now take any $w\in S$. To show that $S\in\mathcal{C}(G)$, we must show that adding $w$ back to $G\setminus S$ reconnects (at least) two vertices lying in separate connected components of $G\setminus S$. We know that adding $w$ back to $H\setminus S$ reconnects at least two vertices $a,b\in N_H(w)\setminus S$ lying in separate connected components of $H\setminus S$ since $S\in\mathcal{C}(H)$. By \cref{PathEquiv}, $a$ and $b$ will lie in separate connected components of $G\setminus S$ also, and will be reconnected in $G\setminus S$ by adding $w$ back to $G\setminus S$. Then $S\in\mathcal{C}(G)$, and the result follows.
\end{proof}

To summarise this \lcnamecref{CVPropsSection}, when $v$ is as in \cref{GoodCVExistence}, we have
\begin{itemize}
	\item $\mathcal{C}(G')=\{S\in\mathcal{C}(G):v\notin S\}$ by \cref{GPrimeSepSets}.
	\item $\mathcal{C}(G'')=\{S\setminus\{v\}:\text{$S\in\mathcal{C}(G)$ with $v\in S$}\}$ by \cref{GoodCVExistence}.
	\item $\mathcal{C}(H)=\{S\in\mathcal{C}(G):\text{$v\notin S$ and $N_G(v)\nsubseteq S$}\}$ by \cref{HSepSets}.
\end{itemize}

\pagebreak

\section{The Main Theorem}

Our goal is to show that, for any block graph $G$, we have
\begin{equation*}\label{BGMainThmIntro}
	\MinAtt_R(H_\mathfrak{m}^i(R/\mathcal{J}_G))=\{\mathfrak{p}\in\Ass_R(R/\mathcal{J}_G):\dim(R/\mathfrak{p})=i\}
\end{equation*}
\label{LRRComment}As noted in \hyperref[Abstract]{the abstract}, this follows from a recent result of Lax, Rinaldo, and Romeo (\cite[Theorem 3.2]{laxSequentiallyCohenMacaulayBinomial2024}), which was published independently during the writing of this paper. We first present our original proof, followed by an \hyperref[MainBGThmAltProof]{alternative proof} using their result.

We must prove a few results beforehand.

\subsection{A Preliminary Lemma}

\begin{proposition}\label{AssFreeExtensionProp}
	Let $R$ be a ring and $z$ an indeterminate. Furthermore, set $S=R[z]$, and let $M$ be an $S$-module such that $zM=0$, viewed also as an $R$-module in the natural way. Then
	\begin{equation*}
		\Ass_S(M)=\{\mathfrak{q}S+(z):\mathfrak{q}\in\Ass_R(M)\}
	\end{equation*}
\end{proposition}

\begin{proof}
	First, note that
	\begin{equation*}
		\Ann_S(N)=\Ann_R(N)S+(z)
	\end{equation*}
	for any $S$-submodule $N$ of $M$ (with $N$ also viewed as an $R$-module in the natural way).

	If $\mathfrak{p}\in\Ass_S(M)$ then $\mathfrak{p}=\Ann_S(m)$ for some $m\in M$. We have that
	\begin{equation*}
		R/\Ann_R(m)\cong S/(\Ann_R(m)S+(z))=S/\mathfrak{p}
	\end{equation*}
	is an integral domain since $\mathfrak{p}\in\Spec(S)$, so $\Ann_R(m)\in\Ass_R(M)$ and therefore 
	\begin{equation*}
		\mathfrak{p}=\Ann_R(M)S+(z)
	\end{equation*}
	is of the desired form.

	Conversely, if $\mathfrak{q}\in\Ass_R(M)$ then $\mathfrak{q}=\Ann_R(m)$ for some $m\in M$. We have that
	\begin{equation*}
		S/\Ann_S(m)=S/(\Ann_R(m)S+(z))\cong R/\Ann_R(M)=R/\mathfrak{q}
	\end{equation*}
	is an integral domain since $\mathfrak{q}\in\Spec(R)$, so
	\begin{equation*}
		\mathfrak{q}S+(z)=\Ann_S(m)\in\Ass_S(M)
	\end{equation*}
	and we are done.
\end{proof}

\begin{lemma}\label{ExtAssLemma}
	Let $R$ be a ring, $\mathfrak{a}$ an ideal of $R$, and $z$ an indeterminate. Set $S=R[z]$, and let $\mathfrak{b}=\mathfrak{a}S+(z)$. Then
	\begin{equation*}
		\Ass_S(\Ext_S^i(S/\mathfrak{b},S))=\{\mathfrak{q}S+(z):\mathfrak{q}\in\Ass_R(\Ext_R^{i-1}(R/\mathfrak{a},R))\}
	\end{equation*}
\end{lemma}

\begin{proof}
	By \cite[Theorem 2.1]{reesTheoremHomologicalAlgebra1956} we have that
	\begin{equation*}
		\Ext_S^i(S/\mathfrak{b},S)\cong\Ext_R^{i-1}(R/\mathfrak{a},R)
	\end{equation*}
	as $S$-modules by setting $A=S$, $\mathfrak{g}=(z)$, $M=S/\mathfrak{b}$ and $N=S$ in their notation, since $S/\mathfrak{b}\cong R/\mathfrak{a}$ as $R$-modules, and $z\Ext_S^i(S/\mathfrak{b},S)=0$, so we are done by \cref{AssFreeExtensionProp}.
\end{proof}

\subsection{Proof of the Main Theorem}\label{MainBGThmProofSection}

\label{ThmEquivChar}By \cref{GradedAssAttDual}, we have
\begin{equation*}
	\Att_R(H_\mathfrak{m}^i(R/\mathcal{J}_G))=\Ass_R(\Ext_R^{2n-i}(R/\mathcal{J}_G,R))
\end{equation*}
and so \hyperref[BGMainThmIntro]{our main theorem} amounts to showing that
\begin{equation*}
	\MinAss_R(\Ext_R^i(R/\mathcal{J}_G,R))=\{\mathfrak{p}\in\Ass_R(R/\mathcal{J}_G):\hgt_R(\mathfrak{p})=i\}
\end{equation*}
since
\begin{equation*}
	\hgt_R(\mathfrak{p})=2n-\dim(R/\mathfrak{p})
\end{equation*}
by \cref{GradedCodim} (which we may apply since every associated prime of $\mathcal{J}_G$ is homogeneous by \cref{HomogAssPrime}).

We first introduce some notation:

\begin{notation}
	For any $R$-module $M$ and $i\geq0$, we set
	\begin{equation*}
		\Ass_R^i(M)\defeq\{\mathfrak{p}\in\Ass_R(M):\hgt_R(\mathfrak{p})=i\}
	\end{equation*}
\end{notation}

\begin{notation}
	For any ideal $\mathfrak{a}$ of $R$ and $i\geq0$, we set
	\begin{equation*}
		E_R^i(\mathfrak{a})\defeq\Ext_R^i(R/\mathfrak{a},R)
	\end{equation*}
\end{notation}

\begin{notation}
	For any $1\leq v\leq n$, we set
	\begin{equation*}
		R_v\defeq k[x_i,y_i:\text{$1\leq i\leq n$ with $i\neq v$}]
	\end{equation*}
\end{notation}

Our key proposition is as follows:

\begin{proposition}\label{ExtAssContainmentProp}
	Let $G$ be a block graph. Then any associated prime of $E_R^i(\mathcal{J}_G)$ contains an associated prime of $R/\mathcal{J}_G$ of height $i$.
\end{proposition}

\begin{proof}
	When $G$ is a disjoint union of cliques, $R/\mathcal{J}_G$ is Cohen-Macaulay (this follows, for example, from \cite[Corollary 2.8]{brunsDeterminantalRings1988} and \cite[Theorem 2.1]{bouchibaTensorProductsCohen2002}). By \cref{GradedTopLCAtt}, the attached primes of $H_\mathfrak{m}^d(R/\mathcal{J}_G)$ are exactly the associated primes of $R/\mathcal{J}_G$ of height $\dim(R/\mathcal{J}_G)$. Since $R/\mathcal{J}_G$ is Cohen-Macaulay in this case, it has no other such local cohomology modules by \cref{GradedCMThm}, and so the result is immediate.

	We will induct first on $n$, the number vertices of $G$. When $n=2$, we must have $G=K_2$, so $G$ is a clique. We have already dealt with this case, and so the base case is established.

	We now induct on the number of cut vertices of $G$. When $G$ has no cut vertices, it must be a disjoint union of cliques. Again, we have already dealt with this case, and so the base case is established for this induction also.

	We may assume then that $G$ contains a cut vertex, and can choose such a cut vertex $v$ of $G$ as in \cref{GoodCVExistence}. We adopt the notation of \cref{BGESSection}.

	The short exact sequence
	\begin{center}
		\begin{tikzcd}
			0 \arrow[r] & R/\mathcal{J}_G \arrow[r,"",hook] & R/Q_1\oplus R/Q_2\arrow[r,"",two heads] & R/Q_3 \arrow[r] & 0
		\end{tikzcd}
	\end{center}
	gives rise to the long exact sequence
	\begin{center}
		\begin{tikzcd}
			\cdots \arrow[r,"\alpha_i"] & E^i_R(Q_1)\oplus E^i_R(Q_2) \arrow[r,""] & E^i_R(\mathcal{J}_G) \arrow[r,""] & E^{i+1}_R(Q_3) \arrow[r,"\alpha_{i+1}"] & \cdots
		\end{tikzcd}
	\end{center}
	Setting
	\begin{equation*}
		A^i\defeq(E_R^i(Q_1)\oplus E_R^i(Q_2))/\im(\alpha_i)
	\end{equation*}
	we then have the short exact sequence
	\begin{center}
		\begin{tikzcd}
			0 \arrow[r] & A^i \arrow[r,"",hook] & E_R^i(\mathcal{J}_G) \arrow[r,"",two heads] & \ker(\alpha_{i+1}) \arrow[r] & 0
		\end{tikzcd}
	\end{center}
	Suppose that $E_R^i(\mathcal{J}_G)\neq0$, and take any $\mathfrak{p}\in\Ass_R(E_R^i(\mathcal{J}_G))$. Localising at $\mathfrak{p}$ gives us the short exact sequence
	\begin{center}
		\begin{tikzcd}
			0 \arrow[r] & A_\mathfrak{p}^i \arrow[r,"",hook] & E_R^i(\mathcal{J}_G)_\mathfrak{p} \arrow[r,"",two heads] & \ker(\alpha_{i+1})_\mathfrak{p} \arrow[r] & 0
		\end{tikzcd}
	\end{center}
	
	If $A_\mathfrak{p}^i=0$, then
	\begin{equation*}
		E_R^i(\mathcal{J}_G)_\mathfrak{p}=\ker(\alpha_{i+1})_\mathfrak{p}\hookrightarrow E_R^{i+1}(Q_3)_\mathfrak{p}
	\end{equation*}
	so $\mathfrak{p}\in\Supp_R(E_R^{i+1}(Q_3))$. Then $\mathfrak{p}$ contains some associated prime of $E_R^{i+1}(Q_3)$, and
	\begin{equation*}
		\Ass_R(E_R^{i+1}(Q_3))=\{\mathfrak{q}R+(x_v,y_v):\mathfrak{q}\in\Ass_{R_v}\hspace{-0.05cm}(E_{R_v}^{i-1}(\mathcal{J}_H))\}
	\end{equation*}
	by \cref{ExtAssLemma} (where $\mathcal{J}_H$ is viewed here as an ideal in $R_v$).

	By the (first) inductive hypothesis, we have
	\begin{equation*}
		\MinAss_{R_v}\hspace{-0.05cm}(E_{R_v}^{i-1}(\mathcal{J}_H))=\{\mathfrak{q}\in\Ass_{R_v}\hspace{-0.05cm}(R_v/\mathcal{J}_H):\hgt_{R_v}\hspace{-0.05cm}(\mathfrak{q})=i-1\}
	\end{equation*}
	We know then that $\mathfrak{p}$ contains $\mathfrak{q}R+(x_v,y_v)$ for some $\mathfrak{q}\in\Ass_{R_v}\hspace{-0.05cm}(R_v/\mathcal{J}_H)$ with $\hgt_{R_v}\hspace{-0.05cm}(\mathfrak{q})=i-1$, and we can write $\mathfrak{q}=P_S(H)$ for some $S\in\mathcal{C}(H)$. We claim that $P_S(G)\subseteq\mathfrak{p}$, and that $\hgt_R(P_S(G))=i$.

	This first claim follows since
	\begin{equation*}
		P_S(G)\subseteq P_S(H)R+(x_v,y_v)
	\end{equation*}
	because the only edges of $G\setminus S$ which may not belong to $H\setminus S$ are of the form $\{v,w\}$ for some $w\in N_G(v)$, and so any elements of $P_S(G)$ introduced by these edges will be included in $(x_v,y_v)$. We will now prove the second claim.

	By \cref{HSepSets}, we have
	\begin{equation*}
		\mathcal{C}(H)=\{S\in\mathcal{C}(G):\text{$v\notin S$ and $N_G(v)\nsubseteq S$}\}
	\end{equation*}
	Then $c_G(S)=c_H(S)$, since $v$ is the only vertex in $G\setminus S$ which is not in $H\setminus S$, and $N_G(v)\nsubseteq S$, so $v$ will belong to the same connected component as at least one of its neighbours in $G\setminus S$.

	By \cref{HeightLemma}, we then have
	\begin{equation*}
		\hgt_R(P_S(G))=\smallAbs{S}+n-c_G(S)=(\smallAbs{S}+(n-1)-c_H(S))+1=\hgt_{R_v}\hspace{-0.05cm}(P_S(H))+1=i
	\end{equation*}
	as desired, and so the result holds in the case that $A^i_\mathfrak{p}=0$.
	
	If $A^i_\mathfrak{p}\neq0$, then we must have either $\mathfrak{p}\in\Supp_R(E_R^i(Q_1))$, or $\mathfrak{p}\in\Supp_R(E_R^i(Q_2))$.

	In the case that $\mathfrak{p}\in\Supp_R(E_R^i(Q_1))$, we have that $\mathfrak{p}$ contains some associated prime of $E_R^i(Q_1)$. Note that $G'$ has fewer cut vertices than $G$ (since we have completed the neighbourhood of $v$, and so it cannot be a cut vertex). Then we can apply the (second) inductive hypothesis to obtain
	\begin{equation*}
		\MinAss_R(E_R^i(Q_3))=\{\mathfrak{q}\in\Ass_R(R/\mathcal{J}_{G'}):\hgt_R(\mathfrak{q})=i\}
	\end{equation*}
	Then $\mathfrak{p}$ contains some $\mathfrak{q}\in\Ass_R(R/\mathcal{J}_{G'})$ with $\hgt_R(\mathfrak{q})=i$, and we can write $\mathfrak{q}=P_S(G')$ for some $S\in\mathcal{C}(G')$. We also have $S\in\mathcal{C}(G)$ by \cref{GPrimeSepSets}, and $P_S(G)\subseteq P_S(G')$ (since we have only added edges to $G$ to obtain $G$'), so we are done with this case if we can show that $\hgt_R(P_S(G))=i$. By \cref{HeightLemma}, this amounts to showing that $c_G(S)=c_{G'}(S)$. This follows from \cref{PathEquiv}, which we may apply since $S\in\mathcal{C}(G')$ and so $v\notin S$ by \cref{GPrimeSepSets}. Then the result holds in this case also.

	Finally, in the case that $\mathfrak{p}\in\Supp_R(E_R^i(Q_2))$, we have that $\mathfrak{p}$ contains some associated prime of $E_R^i(Q_2)$, and
	\begin{equation*}
		\Ass_R(E_R^i(Q_2))=\{\mathfrak{q}R+(x_v,y_v):\mathfrak{q}\in\Ass_{R_v}\hspace{-0.05cm}(E^{i-2}(\mathcal{J}_{G''}))\}
	\end{equation*}
	by \cref{ExtAssLemma} (where $\mathcal{J}_{G''}$ is viewed here as an ideal in $R_v$).

	By the (first) inductive hypothesis, we have
	\begin{equation*}
		\MinAss_{R_v}\hspace{-0.05cm}(E_{R_v}^{i-2}(\mathcal{J}_{G''}))=\{\mathfrak{q}\in\Ass_{R_v}\hspace{-0.05cm}(R_v/\mathcal{J}_{G''}):\hgt_{R_v}\hspace{-0.05cm}(\mathfrak{q})=i-2\}
	\end{equation*}
	We know then that $\mathfrak{p}$ contains $\mathfrak{q}R+(x_v,y_v)$ for some $\mathfrak{q}\in\Ass_{R_v}\hspace{-0.05cm}(R_v/\mathcal{J}_{G''})$ with $\hgt_{R_v}\hspace{-0.05cm}(\mathfrak{q})=i-2$, and we can write $\mathfrak{q}=P_S(G'')$ for some $S\in\mathcal{C}(G'')$.

	Now, we have $S\cup\{v\}\in\mathcal{C}(G)$ (since we have chosen $v$ to be as in \cref{GoodCVExistence}), and
	\begin{equation*}
		P_{S\cup\{v\}}(G)=P_S(G'')R+(x_v,y_v)
	\end{equation*}
	is clearly of the correct height (to see this using \cref{HeightLemma}, note that $G\setminus\{S\cup\{v\}\}=G''\setminus S$), so we are done.
\end{proof}

\pagebreak

We can now prove our main \lcnamecref{MainBGThm}:

\begin{theorem}\label{MainBGThm}
	For any block graph $G$, we have
	\begin{equation*}
		\MinAtt_R(H_\mathfrak{m}^i(R/\mathcal{J}_G))=\{\mathfrak{p}\in\Ass_R(R/\mathcal{J}_G):\dim(R/\mathfrak{p})=i\}
	\end{equation*}
\end{theorem}

\begin{proof}\label{MainBGThmProof}
	By \cite[Algorithm 1.5]{eisenbudDirectMethodsPrimary1992}, for any finitely generated $R$-module $M$, we have
	\begin{equation*}
		\Ass_R^i(M)=\Ass_R^i(\Ext_R^i(M,R))
	\end{equation*}
	Combining this with \cref{ExtAssContainmentProp} completes the proof.
\end{proof}

\begin{corollary}\label{MainBGCorr}
	Let $G$ be a block graph. Then we have that $H_\mathfrak{m}^i(R/\mathcal{J}_G)\neq0$ if and only if $R/\mathcal{J}_G$ has an associated prime of dimension $i$. This can be checked combinatorially via \cref{HeightLemma} (noting \cref{GradedCodim}).
\end{corollary}

\begin{note}
	\cref{MainBGCorr} can be seen as a generalisation of the main result of \cite[Theorem 1.1]{eneCohenMacaulayBinomialEdge2011} which, by \cref{GradedLCDepthThm}, is equivalent to stating that, for a block graph $G$ on $n$ vertices, we have
	\begin{equation*}
		\min\{i\geq0:H_\mathfrak{m}^i(R/\mathcal{J}_G)\neq0\}=n+c
	\end{equation*}
	where $c$ is the number of connected components of $G$. It is easily seen that $n+c$ is the height of $P_\varnothing(G)$, and that this height is minimal amongst the heights of the associated primes of $R/\mathcal{J}_G$, so applying \cref{MainBGCorr} yields this result.

	Whilst block graphs are not explicitly given as the family of graphs in the statement of \cite[Theorem 1.1]{eneCohenMacaulayBinomialEdge2011}, it can been seen that these families are the same via, for example, \cite[Theorem 2.6]{howorkaMetricPropertiesCertain1979}.
\end{note}

\subsection{An Alternative Proof of the Main Theorem}

We first describe a generalisation of Cohen-Macaulay $R$-modules first introduced by Stanley:

\begin{definition}\thmCite[p.~87, Definition 2.9]{stanleyCombinatoricsCommutativeAlgebra1996}\label{SeqCMDef}
	Let $k$ be a field, and $R$ a non-negatively $\mathbb{Z}$-graded $k$-algebra of finite type such that the homogeneous elements of $R$ of degree $0$ are precisely $k$. We say that a finitely generated $\mathbb{Z}$-graded $R$-module $M$ is \emph{sequentially Cohen-Macaulay} if there exists a filtration of graded $R$-modules
	\begin{equation*}
		0=M_0\subseteq M_1\subseteq\cdots\subseteq M_{t-1}\subseteq M_t=M
	\end{equation*}
	for some $t\geq0$ such that:
	\begin{enumerate}
		\item For each $1\leq i\leq t$, we have that $M_i/M_{i-1}$ is Cohen-Macaulay.
		\item We have
		\begin{equation*}
			\dim_R(M_j/M_{j-1})<\dim_R(M_{j+1}/M_j)
		\end{equation*}
		for each $1\leq j\leq t-1$.
	\end{enumerate}
\end{definition}

An important alternative characterisation of sequentially Cohen-Macaulay $R$-modules was given by Peskine. We state it here in the specific case of a polynomial ring over a field:
\begin{theorem}\label{SeqCMAltChar}
	Let $R=k[x_1,\ldots,x_n]$ for some field $k$ and $n\geq0$ with the standard $\mathbb{Z}$-grading, and let $M$ be a finitely generated $\mathbb{Z}$-graded $R$-module of dimension $d$. Then $M$ is sequentially Cohen-Macaulay if and only if, for all $0\leq i\leq d$, we have that $\Ext_R^{n-i}(M,R)$ is either $0$ or Cohen-Macaulay of dimension $i$.
\end{theorem}

\begin{proof}
	This follows from \cite[Theorem 1.4]{herzogSequentiallyCohenMacaulayModules2002} and \cref{GradedGorenCanonMod} since $R$ is Gorenstein.
\end{proof}

Characterising families of graphs which give rise to sequentially Cohen-Macaulay binomial edge ideals is an active area of study. During the writing of this thesis, Lax, Rinaldo, and Romeo showed in \cite{laxSequentiallyCohenMacaulayBinomial2024} that several well-known classes of graphs have sequentially Cohen-Macaulay binomial edge ideals. In particular:
\begin{theorem}\thmCite[Theorem 3.2]{laxSequentiallyCohenMacaulayBinomial2024}\label{BGSeqCM}
	For any block graph $G$, $R/\mathcal{J}_G$ is sequentially Cohen-Macaulay.
\end{theorem}
As was \hyperref[LRRComment]{mentioned} at the start of this section, this \lcnamecref{BGSeqCM} implies \cref{MainBGThm}. In fact, it implies the following stronger result:

\begin{theorem}\label{BGLCAtt}
	For any block graph $G$, we have
	\begin{equation*}
		\Att_R(H_\mathfrak{m}^i(R/\mathcal{J}_G))=\{\mathfrak{p}\in\Ass_R(R/\mathcal{J}_G):\dim(R/\mathfrak{p})=i\}
	\end{equation*}
\end{theorem}

We will prove \cref{BGLCAtt} using \cref{BGSeqCM} and the \hyperref[BGThmImplication]{following \lcnamecref{BGThmImplication}}:

\begin{proposition}\label{BGThmImplication}
	Let $R=k[x_1,\ldots,x_n]$ for some field $k$ and $n\geq0$ with the standard $\mathbb{Z}$-grading, and let $M$ be a finitely generated $\mathbb{Z}$-graded $R$-module. Then, if $M$ is sequentially Cohen-Macaulay, we have 
	\begin{equation*}
		\Ass_R(\Ext_R^i(M,R))=\Ass_R^i(M)
	\end{equation*}
\end{proposition}

\begin{proof}
	As noted in \hyperref[MainBGThmProof]{the proof} of \cref{MainBGThm}, we have
	\begin{equation*}
		\Ass_R^i(\Ext_R^i(M,R))=\Ass_R^i(M)
	\end{equation*}
	by \cite[Algorithm 1.5]{eisenbudDirectMethodsPrimary1992}. Then if $\Ext_R^i(M,R)=0$, we must have $\Ass_R^i(M)=\varnothing$, so the result is trivially true in this case.

	Otherwise, by \cref{SeqCMAltChar} we have that $\Ext_R^i(M,R)$ is Cohen-Macaulay of dimension $n-i$, so every associated prime of $\Ext_R^i(M,R)$ is of height $n-(n-i)=i$ by \cref{GradedCMUnmixed}, and we are done.
\end{proof}

\begin{customProof}[Proof of \cref{BGLCAtt}]\label{MainBGThmAltProof}
	By \cref{BGSeqCM}, we may apply \cref{BGThmImplication} to $R/\mathcal{J}_G$. As we \hyperref[ThmEquivChar]{noted earlier}, we have that
	\begin{equation*}
		\Att_R(H_\mathfrak{m}^i(R/\mathcal{J}_G))=\Ass_R(\Ext_R^{2n-i}(R/\mathcal{J}_G,R))
	\end{equation*}
	and so the result follows.
\end{customProof}

\section{Some Counterexamples}

\cref{MainBGThm} is very far from true in general, as computations in \citeMacaulay{} (taking $k=\mathbb{Z}/2\mathbb{Z}$) show:

When $n=5$ and
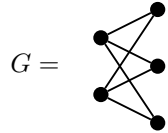
\begin{equation*}
    G=\quad\begin{tikzpicture}[x=0.75cm,y=0.75cm,every node/.style={circle,draw=black,fill=black,inner sep=0pt,minimum size=5pt},label distance=0.15cm,line width=0.25mm,baseline={([yshift=-0.5ex]current bounding box.center)}]
        \node(1) at (0,1.5) {};
		\node(2) at (0,0.5) {};
		\node(3) at (1,2) {};
		\node(4) at (1,1) {};
		\node(5) at (1,0) {};

	    \foreach \from/\to in {1/3,1/4,1/5,2/3,2/4,2/5}
        	\draw[-] (\from) -- (\to);
    \end{tikzpicture}
\end{equation*}
we have $H_\mathfrak{m}^5(R/\mathcal{J}_G)\neq0$, but $R/\mathcal{J}_G$ has no associated prime of dimension 5.

\label{EmdAttExample}We can also have embedded attached primes. For example, let $n=8$ and
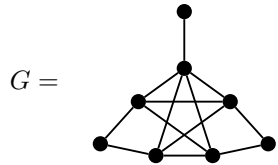
\begin{equation*}
    G=\quad\begin{tikzpicture}[x=0.75cm,y=0.75cm,every node/.style={circle,draw=black,fill=black,inner sep=0pt,minimum size=5pt},label distance=0.15cm,line width=0.25mm,baseline={([yshift=-0.5ex]current bounding box.center)}]
        \node(1) at (0,0) {};
		\node(2) at ([shift=(0:1)]1) {};
		\node(3) at ([shift=(72:1)]2) {};
		\node(4) at ([shift=(144:1)]3) {};
		\node(5) at ([shift=(216:1)]4) {};

		\node(6) at ([shift=(168:1)]1) {};
		\node(7) at ([shift=(12:1)]2) {};

		\node(8) at ([shift=(90:1)]4) {};

	    \foreach \from/\to in {1/2,1/3,1/4,1/6,2/4,2/5,2/7,3/4,3/5,3/7,4/5,4/8,5/6}
        	\draw[-] (\from) -- (\to);
    \end{tikzpicture}
\end{equation*}
Then \citeMacaulay{} tells us that there are $\mathfrak{p},\mathfrak{q}\in\Spec(R)$ such that
\begin{equation*}
	\Att_R(H_\mathfrak{m}^7(R/\mathcal{J}_G))=\{\mathfrak{p},\mathfrak{q}\}
\end{equation*}
with $\mathfrak{q}\subsetneq\mathfrak{p}$ and $\dim(R/\mathfrak{p})=7$, so $H_\mathfrak{m}^7(R/\mathcal{J}_G)$ has an embedded attached prime, and the minimal attached prime of $H_\mathfrak{m}^7(R/\mathcal{J}_G)$ is not of dimension $7$.

\pagebreak

\appendix

\section*{Appendix - Some Graded Analogues of Local Results}\label{GradedAppendix}
\renewcommand{\thesubsection}{A.\arabic{subsection}}
\setcounter{theorem}{0}
\renewcommand{\thetheorem}{A.\arabic{theorem}}

In this appendix, we set $R=k[x_1,\ldots,x_n]$ for some field $k$ and $n\geq0$, and $\mathfrak{m}=(x_1,\ldots,x_n)$.

Most of the results presented here are well known, however we include them in this appendix for completeness.

Our first goal is to show that local cohomology modules of finitely generated modules supported at homogeneous ideals have secondary representations in which each secondary component is graded and each attached prime is homogeneous.

We begin by noting that several natural constructions involving graded modules are also graded:

\begin{proposition}\thmCite[pp. 116-117]{northcottLessonsRingsModules1968}
	Let $M$ be a graded $R$-module, and $N$ a graded $R$-submodule of $M$. Then $M/N$ is also graded.
\end{proposition}

\begin{proposition}\label{GradedExt}
	Let $M$ be a finitely generated graded $R$-module, $N$ a graded $R$-module, $\mathfrak{a}$ a homogeneous ideal of $R$, and $i\geq0$. Then $\Ext_R^i(M,N)$ and $H_\mathfrak{a}^i(M)$ are also graded.
\end{proposition}

\begin{proof}
	See \cite[pp. 32--33]{brunsCohenMacaulayRings2005} and \cite[Corollary 13.4.3]{brodmannLocalCohomologyAlgebraic2013}.
\end{proof}

Next, we show that such local cohomology modules are Artinian:

\begin{proposition}\label{GradedLCArt}
	Let $M$ be a finitely generated graded $R$-module. Then $H_\mathfrak{m}^i(M)$ is Artinian for any $i\geq0$.
\end{proposition}

\begin{proof}
	If $H_\mathfrak{m}^i(M)=0$ then the result is trivial, so suppose not. Take any descending chain of $R$-modules
	\begin{equation*}
		H_\mathfrak{m}^i(M)=A_0\supseteq A_1\supseteq\cdots
	\end{equation*}
	Then
	\begin{equation*}
		H_{\mathfrak{m}R_\mathfrak{m}}^i\hspace{-0.05cm}(M_\mathfrak{m})\cong H_\mathfrak{m}^i(M)_\mathfrak{m}=(A_0)_\mathfrak{m}\supseteq(A_1)_\mathfrak{m}\supseteq\cdots
	\end{equation*}
	and so since $H_{\mathfrak{m}R_\mathfrak{m}}^i\hspace{-0.05cm}(M_\mathfrak{m})$ is Artinian there exists some $t\geq0$ such that
	\begin{equation*}
		(A_t/A_{t+j})_\mathfrak{p}\cong(A_t)_\mathfrak{p}/(A_{t+j})_\mathfrak{p}=0
	\end{equation*}
	for all $j\geq0$. Then the result follows since $H_\mathfrak{m}^i(M)_\mathfrak{n}=0$ for any maximal ideal $\mathfrak{n}$ of $R$ which is not equal to $\mathfrak{m}$.
\end{proof}

We will also make use of the following key lemma:

\begin{lemma}\label{ArtGradedSecRep}
	Let $R$ be an Artinian graded $R$-module. Then $M$ has a minimal secondary representation with each secondary component graded and each attached prime homogeneous.
\end{lemma}

\begin{proof}
	This follows from \cite[Proposition 2.4]{sharpAsymptoticBehaviourCertain1986} and \cite[Corollary 1.6]{richardsonAttachedPrimesGraded2003}.
\end{proof}

Our first claim is then immediate:

\begin{corollary}\label{GradedSecRep}
	Let $M$ be a finitely generated graded $R$-module. Then for any $i\geq0$, $H_\mathfrak{m}^i(M)$ has a minimal secondary representation with each secondary component graded and each attached prime homogeneous.
\end{corollary}

\begin{proof}
	This follows from \cref{GradedLCArt} and \cref{ArtGradedSecRep}.
\end{proof}

Next, we aim to prove an equality between the attached primes of these local cohomology modules and the associated primes of certain $\Ext$ modules.

We will first examine the effect of localisation on secondary representations.

\begin{note}
	\cref{SecondaryLoc} does not make any use of grading.
\end{note}

\pagebreak

\begin{proposition}\label{SecondaryLoc}
	Let $M$ be a $\mathfrak{q}$-secondary $R$-module for some $\mathfrak{p}\in\Spec(R)$, and let $\mathfrak{p}\in\Spec(R)$ with $\mathfrak{q}\subseteq\mathfrak{p}$. Then $M_\mathfrak{p}$ is either $0$, or $\mathfrak{q}R_\mathfrak{p}$-secondary as an $R_\mathfrak{p}$-module.
\end{proposition}

\begin{proof}
	Suppose that $M\neq0$, and take any $\frac{q}{s}\in\mathfrak{q}R_\mathfrak{p}$. Since $M$ is $\mathfrak{q}$-secondary, there exists some $l\geq1$ such that $q^l$ annihilates $M$, so certainly $(\frac{q}{s})^l$ annihilates $M_\mathfrak{p}$.

	Next, take any $\frac{r}{s}\in R_\mathfrak{p}\setminus\mathfrak{q}R_\mathfrak{p}$, and any $\frac{m}{s'}\in M_\mathfrak{p}$. Note that $s',r\notin\mathfrak{q}$, so $s'r\notin\mathfrak{q}$. Then, again since $M$ is $\mathfrak{q}$-secondary, there exists some $m'\in M$ such that $s'rm'=sm$, and so $\frac{r}{s}\cdot\frac{m'}{1}=\frac{m}{s'}$. This shows that $\frac{r}{s}M_\mathfrak{p}=M_\mathfrak{p}$, so we are done.
\end{proof}

In many cases, these localisations will be $0$. For example, any prime in the support of an Artinian $R$-module must be maximal (see, for example, \cite[Lemma 3.1.11 (1)]{leamerHomologyArtinianModules2011}). However, we can say \hyperref[AttLoc]{something useful} in the graded case. In doing so, we will make use of the following lemma:

\begin{lemma}\label{GradedLocFF}
	For any graded $R$-module $M$, we have $M_\mathfrak{m}=0$ if and only if $M=0$.
\end{lemma}

\begin{proof}
	This follows from \cite[Proposition 1.5.15~(c)]{brunsCohenMacaulayRings2005}.
\end{proof}

\begin{proposition}\label{AttLoc}
	Let be $M$ a representable $R$-module. Then
	\begin{equation*}
		\Att_{R_\mathfrak{m}}\hspace{-0.05cm}(M_\mathfrak{m})=\{\mathfrak{q}R_\mathfrak{m}:\mathfrak{q}\in\Att_R(M)\}
	\end{equation*}
\end{proposition}

\begin{proof}
	By \cref{ArtGradedSecRep}, $M$ has minimal a secondary representation
	\begin{equation*}
		M=A_1+\cdots+A_t
	\end{equation*}
	for some $t\geq1$, with each $A_i$ graded and $\mathfrak{q}_i$-secondary for some homogeneous $\mathfrak{q}_i\in\Spec(R)$. In particular, $\mathfrak{q}_i\subseteq\mathfrak{m}$. Then
	\begin{equation*}
		M_\mathfrak{m}=(A_1+\cdots+A_t)_\mathfrak{m}=(A_1)_\mathfrak{m}+\cdots+(A_t)_\mathfrak{m}
	\end{equation*}
	and we are done by \cref{GradedLocFF} and \cref{SecondaryLoc}.
\end{proof}

A key tool used in the remainder of this appendix is the following result on dimension:

\begin{proposition}\thmCite[Exercise 1.4.14]{brodmannLocalCohomologyAlgebraic2013}\label{GradedDim}
	Let $M$ be a graded $R$-module. Then
	\begin{equation*}
		\dim_R(M)=\dim_{R_\mathfrak{m}}\hspace{-0.05cm}(M_\mathfrak{m})
	\end{equation*}
\end{proposition}

We will also need the following well-known result on the associated primes of graded modules:

\begin{proposition}\thmCite[Lemma 1.5.6(b)]{brunsCohenMacaulayRings2005}\label{HomogAssPrime}
	Let $M$ be a graded $R$-module. Then each $\mathfrak{p}\in\Ass_R(M)$ is homogeneous.
\end{proposition}

We can now prove the following equality:

\begin{lemma}\label{GradedAssAttDual}
	Let $M$ be a finitely generated graded $R$-module. Then
	\begin{equation*}
		\Att_R(H_\mathfrak{m}^i(M))=\Ass_R(\Ext_R^{n-i}(M,R))
	\end{equation*}
\end{lemma}

\begin{proof}
	By the usual Local Duality, \cite[Corollary 10.2.20]{brodmannLocalCohomologyAlgebraic2013}, and \cref{GradedDim}, we have
	\begin{align*}
		\Att_{R_\mathfrak{m}}\hspace{-0.05cm}(H_{\mathfrak{m}R_\mathfrak{m}}^i\hspace{-0.05cm}(M_\mathfrak{m}))&=\Ass_{R_\mathfrak{m}}\hspace{-0.05cm}(\Ext_{R_\mathfrak{m}}^{n-i}(M_\mathfrak{m},R_\mathfrak{m}))\\
		&=\Ass_{R_\mathfrak{m}}\hspace{-0.05cm}(\Ext_R^{n-i}(M,R)_\mathfrak{m})\\
		&=\{\mathfrak{p}R_\mathfrak{m}:\text{$\mathfrak{p}\in\Ass_R(\Ext_R^{n-i}(M,R))$ with $\mathfrak{p}\subseteq\mathfrak{m}$}\}\\
		&=\{\mathfrak{p}R_\mathfrak{m}:\mathfrak{p}\in\Ass_R(\Ext_R^{n-i}(M,R))\}
	\end{align*}
	with the final equality following since every associated prime of $\Ext_R^{n-i}(M,R)$ is homogeneous by \cref{HomogAssPrime} since it is graded by \cref{GradedExt}.

	Using \cref{AttLoc}, we also obtain
	\begin{equation*}
		\Att_{R_\mathfrak{m}}\hspace{-0.05cm}(H_{\mathfrak{m}R_\mathfrak{m}}^i\hspace{-0.05cm}(M_\mathfrak{m}))=\Att_{R_\mathfrak{m}}\hspace{-0.05cm}(H_\mathfrak{m}^i(M)_\mathfrak{m})=\{\mathfrak{p}R_\mathfrak{m}:\mathfrak{p}\in\Att_R(H_\mathfrak{m}^i(M))\}
	\end{equation*}
	and so we are done.
\end{proof}

We can also describe the attached primes of the top local cohomology module of a graded module in terms of the associated primes of the module itself by proving a graded version of \cite[Theorem 2.2]{macdonaldElementaryProofNonVanishing1972}:

\begin{theorem}\label{GradedTopLCAtt}
	Let $M$ be a finitely generated graded $R$-module of dimension $d$. Then
	\begin{equation*}
		\Att_R(H_\mathfrak{m}^d(M))=\{\mathfrak{p}\in\Ass_R(M):\dim(R/\mathfrak{p})=d\}
	\end{equation*}
\end{theorem}

\begin{proof}
	By \cite[Theorem 2.2]{macdonaldElementaryProofNonVanishing1972}, \cref{HomogAssPrime}, and \cref{GradedDim} we have
	\begin{align*}
		\Att_{R_\mathfrak{m}}\hspace{-0.05cm}(H_{\mathfrak{m}R_\mathfrak{m}}^d\hspace{-0.05cm}(M_\mathfrak{m}))&=\{\mathfrak{p}R_\mathfrak{m}\in\Ass_{R_\mathfrak{m}}\hspace{-0.05cm}(M_\mathfrak{m}):\dim(R_\mathfrak{m}/\mathfrak{p}R_\mathfrak{m})=d\}\\
		&=\{\mathfrak{p}R_\mathfrak{m}:\text{$\mathfrak{p}\in\Ass_R(M)$ with $\dim(R/\mathfrak{p})=d$}\}
	\end{align*}
	and we saw in the proof of \cref{GradedAssAttDual} that
	\begin{equation*}
		\Att_{R_\mathfrak{m}}\hspace{-0.05cm}(H_{\mathfrak{m}R_\mathfrak{m}}^i\hspace{-0.05cm}(M_\mathfrak{m}))=\{\mathfrak{p}R_\mathfrak{m}:\mathfrak{p}\in\Att_R(H_\mathfrak{m}^i(M))\}
	\end{equation*}
	for any $i\geq0$, which completes the proof.
\end{proof}

It is well-known that for Cohen-Macaulay local rings, the height of an ideal is equal to its codimension. We have the following graded analogue:

\begin{lemma}\label{GradedCodim}
Let $\mathfrak{a}$ be a homogeneous ideal of $R$. Then
	\begin{equation*}
		\hgt_R(\mathfrak{a})=n-\dim(R/\mathfrak{a})
	\end{equation*}
\end{lemma}

\begin{proof}
	We have
	\begin{equation*}
		\hgt_R(\mathfrak{a})\overset{\text(1)}{=}\grade_R(\mathfrak{a},R)\overset{\text(2)}{=}\grade_{R_\mathfrak{m}}\hspace{-0.05cm}(\mathfrak{a}_\mathfrak{m},R_\mathfrak{m})\overset{\text{(3)}}{=}\dim_{R_\mathfrak{m}}\hspace{-0.05cm}(R_\mathfrak{m})-\dim_{R_\mathfrak{m}}\hspace{-0.05cm}(R_\mathfrak{m}/\mathfrak{a}_\mathfrak{m})\overset{\text{(4)}}{=}n-\dim(R/\mathfrak{a})
	\end{equation*}
	where (1) follows from \cite[Corollary 2.1.4]{brunsCohenMacaulayRings2005}, (2) follows from \cite[Proposition 1.5.15(e)]{brunsCohenMacaulayRings2005}, (3) follows from \cite[Theorem 2.1.2]{brunsCohenMacaulayRings2005} since $R$, and therefore $R_\mathfrak{m}$, is Cohen-Macaulay, and (4) follows from \cref{GradedDim}.
\end{proof}

It is also well-known that a local ring $(R,\mathfrak{m})$, an $R$-module $M$ is Cohen-Macaulay if and only if $H_\mathfrak{m}^i(M)$ is non-zero for $i=\dim_R(M)$, and vanishes otherwise. Furthermore, if $M$ is Cohen-Macaulay, then its associated primes are all of the same height. We will show that similar results hold in the our setting.

To do this, we will need a result which is familiar in the local case:

\begin{proposition}\thmCite[Exercise 2.1.27(c)]{brunsCohenMacaulayRings2005}\label{GradedLocCM}
	Let $M$ be a finitely generated graded $R$-module. Then $M$ is Cohen-Macaulay as an $R$-module if and only if $M_\mathfrak{m}$ is Cohen-Macaulay as an $R_\mathfrak{m}$-module.
\end{proposition}

With this in hand, we can prove our claims:

\begin{lemma}\label{GradedCMUnmixed}
	Let $M$ be a finitely generated Cohen-Macaulay graded $R$-module. Then all associated primes of $M$ are of the same height.
\end{lemma}

\begin{proof}
	Every associated prime of $M$ is homogeneous since it is graded, and so is contained in $\mathfrak{m}$. Then
	\begin{equation*}
		\Ass_{R_\mathfrak{m}}\hspace{-0.05cm}(M_\mathfrak{m})=\{\mathfrak{p}R_\mathfrak{m}:\text{$\mathfrak{p}\in\Ass_R(M)$ with $\mathfrak{p}\subseteq\mathfrak{m}$}\}=\{\mathfrak{p}R_\mathfrak{m}:\mathfrak{p}\in\Ass_R(M)\}
	\end{equation*}
	Since $M$ is a Cohen Macaulay $R$-module, $M_\mathfrak{m}$ is a Cohen-Macaulay $R_\mathfrak{m}$-module, so every associated prime has the same height by \cite[Theorem 17.3(i)]{matsumuraCommutativeRingTheory1986}, and so we are done by applying \cite[Corollary 1.5.8(a)]{brunsCohenMacaulayRings2005} to $R$.
\end{proof}

\begin{theorem}\label{GradedCMThm}
	Let $M$ be a finitely generated graded $R$-module of dimension $d$. Then $M$ is Cohen-Macaulay if and only if $H_\mathfrak{m}^i(M)$ is non-zero for $i=d$ and vanishes otherwise.
\end{theorem}

\begin{proof}
	By \cref{GradedLocCM}, $M$ is Cohen-Macaulay if and only if $M_\mathfrak{m}$ is Cohen-Macaulay as an $R_\mathfrak{m}$-module, which by the usual result on local rings holds if and only if $H_{\mathfrak{m}_\mathfrak{m}}^d\hspace{-0.05cm}(M_\mathfrak{m})\neq0$ and vanishes otherwise, since $\dim_{R_\mathfrak{m}}\hspace{-0.05cm}(M_\mathfrak{m})=d$ by \cref{GradedDim}. Then we are done by \cref{GradedLocFF}.
\end{proof}

\pagebreak

We can also characterise depth (meaning grade with respect to $\mathfrak{m}$) as in the local case:

\begin{theorem}\label{GradedLCDepthThm}
	Let $M$ be a finitely generated graded $R$-module. Then
	\begin{equation*}
		\depth_R(M)=\min\{i\geq0:H_\mathfrak{m}^i(M)\neq0\}
	\end{equation*}
\end{theorem}

\begin{proof}
	This follows from \cite[Theorem 6.2.7]{brodmannLocalCohomologyAlgebraic2013} and the graded version of Nakyama's Lemma (see, for example, \cite[Exercise 1.5.24 (a)]{brunsCohenMacaulayRings2005}, taking $N=0$).
\end{proof}

Finally, a well-known property of Gorenstein local rings holds in our setting also:

\begin{lemma}\label{GradedGorenCanonMod}
	$R$ is a canonical module for itself.
\end{lemma}

\begin{proof}
	This follows from \cite[Corollary 14.5.16]{brodmannLocalCohomologyAlgebraic2013} since $R$ is Gorenstein.
\end{proof}

\pagebreak

\nocite{M2}
\printbibliography

\end{document}

%% file: preamble.tex
\usepackage[T1]{fontenc}
\usepackage[utf8]{inputenc}
\usepackage[style=alphabetic,sorting=nyt,giveninits=true,backend=biber]{biblatex}
\usepackage{geometry}
\usepackage{parskip}
\usepackage{amsmath}
\usepackage{amssymb}
\usepackage{amsthm}
\usepackage{hyperref}
\usepackage[capitalize,nameinlink,noabbrev]{cleveref}
\usepackage{mathtools}
\usepackage{commath}
\usepackage{bm}
\usepackage{tikz}
\usetikzlibrary{cd}
\usepackage{enumitem}
\usepackage{tcolorbox}
\tcbuselibrary{skins}
\usepackage{xparse}
\usepackage{microtype}

\makeatletter
\def\thmCite{\@ifnextchar[{\@with}{\@without}}
\def\@with[#1]#2{{\normalfont\cite[#1]{#2}\;}}
\def\@without#1{{\normalfont\cite{#1}\;}}
\makeatother

\newenvironment{customProof}[1][Proof]{\let\oldProofName\proofname\renewcommand*{\proofname}{#1}\begin{proof}}{\end{proof}\renewcommand*{\proofname}{\oldProofName}}

\ExplSyntaxOn
\DeclareExpandableDocumentCommand{\IfNoValueOrEmptyTF}{mmm}
 {
  \IfNoValueTF{#1}{#2}
   {
    \tl_if_empty:nTF {#1} {#2} {#3}
   }
 }
\ExplSyntaxOff

\newcommand*{\valOrBlank}[1]{\IfNoValueOrEmptyTF{#1}{}{#1}}

\NewDocumentCommand{\mySum}{e{_^}}{%
	\sum_{\mathclap{\valOrBlank{#1}}}^{\mathclap{\valOrBlank{#2}}}\hspace{0.025cm}%
}

\NewDocumentCommand{\myCap}{e{_^}}{%
	\bigcap_{\mathclap{\valOrBlank{#1}}}^{\mathclap{\valOrBlank{#2}}}\hspace{0.05cm}%
}

\makeatletter
\newcommand*{\defeq}{\hspace{-0.04cm}\mathrel{\rlap{%
                     \raisebox{0.4ex}{$\scriptstyle\m@th\cdot$}}%
                     \raisebox{-0.05ex}{$\scriptstyle\m@th\cdot$}}%
				 =}
\makeatother

\ExplSyntaxOn
  \cs_new:Npn \__get_max_length:nN #1#2 {
    \dim_set:Nn \l_tmpa_dim { 0pt }
    \clist_map_inline:nn { #1 } {
      \hbox_set:Nn \l_tmpa_box {##1}
      \dim_compare:nT {\l_tmpa_dim < \box_wd:N \l_tmpa_box}
      {
        \dim_set:Nn \l_tmpa_dim { \box_wd:N \l_tmpa_box }
      } 
    }
    \dim_set_eq:NN #2 \l_tmpa_dim
  }

  \NewDocumentCommand{\maxlength}{}{\__get_max_length:nN}
\ExplSyntaxOff

\makeatletter
\newcommand{\myitem}[1]{%
\item[#1]\protected@edef\@currentlabel{#1}%
}
\makeatother

\tcbset{sharp corners, colback=white, colframe=black, boxrule=0.4pt,parbox=false}

\begingroup
\makeatletter
\@for\theoremstyle:=definition,remark,plain\do{%
	\expandafter\g@addto@macro\csname th@\theoremstyle\endcsname{%
		\addtolength\thm@preskip\parskip
	}%
}
\endgroup

\let\oldproof\proof
\def\proof{\oldproof\unskip}

\makeatletter
\def\footnoterule{
  \hrule \@width 9cm \kern 3\p@} 
\makeatother

\DeclareMathOperator{\Ann}{Ann}
\DeclareMathOperator{\Ass}{Ass}
\DeclareMathOperator{\Att}{Att}

\DeclareMathOperator{\depth}{depth}
\DeclareMathOperator{\Ext}{Ext}
\DeclareMathOperator{\hgt}{height}

\DeclareMathOperator{\im}{Im}
\DeclareMathOperator{\grade}{grade}
\DeclareMathOperator{\MinAss}{MinAss}
\DeclareMathOperator{\MinAtt}{MinAtt}
\DeclareMathOperator{\Spec}{Spec}
\DeclareMathOperator{\Supp}{Supp}

\DeclarePairedDelimiter{\smallAbs}{|}{|}

\newtheorem{theorem}{Theorem}[section]
\newtheorem{lemma}[theorem]{Lemma}
\newtheorem{corollary}[theorem]{Corollary}
\newtheorem{proposition}[theorem]{Proposition}
\newtheorem{definition}[theorem]{Definition}
\newtheorem{notation}[theorem]{Notation}

\newtheorem{example}[theorem]{Example}
\theoremstyle{remark}
\newtheorem*{note}{Note}

\newcommand{\code}[1]{{\normalfont\ttfamily #1}}

\newcommand{\blockComment}[1]{}

\DeclareTextFontCommand{\emph}{\boldmath\bfseries}

\newcommand{\citelink}[2]{{\hypersetup{linkbordercolor={0 1 0}}\hyperlink{cite.\therefsection @#1}{#2}}}
\newcommand{\citeMacaulay}{\citelink{M2}{\code{Macaulay2}}}

\DeclareFieldFormat{postnote}{#1}
\DeclareFieldFormat{multipostnote}{#1}

\AtBeginBibliography{\vspace*{4pt}}

\DeclareFieldFormat[article,book,incollection,inproceedings,misc]{title}{\mkbibitalic{#1}}
\DeclareFieldFormat[inproceedings]{booktitle}{\normalfont{#1}}
\DeclareFieldFormat[article]{journaltitle}{#1\isdot}
\DeclareFieldFormat[article,incollection]{booktitle}{\normalfont{#1}}